\documentclass[11pt,reqno]{amsart}

\usepackage[a4paper,margin=1in]{geometry}
\usepackage{amsmath,amssymb,amsthm,mathtools}
\usepackage[expansion=false]{microtype}
\usepackage{aliascnt}
\usepackage{graphicx}
\usepackage[hidelinks]{hyperref}
\usepackage[capitalize]{cleveref}

\newtheorem{theorem}{Theorem}[section]
\newaliascnt{proposition}{theorem}
\newtheorem{proposition}[proposition]{Proposition}
\aliascntresetthe{proposition}
\newaliascnt{lemma}{theorem}
\newtheorem{lemma}[lemma]{Lemma}
\aliascntresetthe{lemma}
\newaliascnt{corollary}{theorem}

\aliascntresetthe{corollary}
\theoremstyle{definition}
\newaliascnt{definition}{theorem}
\newtheorem{definition}[definition]{Definition}
\aliascntresetthe{definition}
\theoremstyle{remark}
\newaliascnt{remark}{theorem}
\newtheorem{remark}[remark]{Remark}
\aliascntresetthe{remark}

\newcommand{\R}{\mathbb{R}}

\newcommand{\Z}{\mathbb{Z}}
\newcommand{\C}{\mathbb{C}}
\newcommand{\T}{\mathbb{T}}
\newcommand{\PP}{\mathbb{P}}
\newcommand{\EE}{\mathbb{E}}
\newcommand{\F}{\mathcal{F}}
\newcommand{\abs}[1]{\left|#1\right|}

\newcommand{\one}{\mathbf 1}
\newcommand{\Mz}{\mathtt{M}}
\newcommand{\Sinv}{\widehat S^{\,\mathtt{inv}}}
\DeclareMathOperator{\Var}{Var}
\DeclareMathOperator{\Unif}{Unif}
\DeclareMathOperator{\Imag}{Im}
\newcommand{\blfootnote}[1]{%
  \begingroup
  \renewcommand\thefootnote{}\footnote{#1}%
  \addtocounter{footnote}{-1}%
  \endgroup
}

\title{Sharp asymptotics for the tree-completion time in\\ cylindrical Hastings--Levitov$(0)$}
\author[Xiao-Ming Fu]{Xiao-Ming Fu}
\address{School of Mathematical Sciences, University of Science and Technology of China,
Hefei, Anhui 230026, P.R.\ China}
\email{fuxm@ustc.edu.cn}
\author[Tianyang Sun]{Tianyang Sun}
\address{School of Mathematical Sciences, University of Science and Technology of China,
Hefei, Anhui 230026, P.R.\ China}
\email{tysun@mail.ustc.edu.cn}
\author[Yuxuan Zong]{Yuxuan Zong}
\address{School of Mathematical Sciences, Peking University, Beijing, China}
\email{yxzong25@stu.pku.edu.cn}
\date{July 2026}

\begin{document}
\begin{abstract}
Let $\mathrm{CHL}_N$ be the cylindrical Hastings--Levitov aggregation process with
parameter $0$ on a cylinder of width $N$ with particles of fixed size $\lambda>0$, and let
$\omega_{N,\lambda}$ be its \emph{tree-completion time} --- the last time at which a new tree is
born on the base circle. Chen, Procaccia and Zong proved the sharp upper bound
$\EE[\omega_{N,\lambda}]\le(1+\varepsilon)(\log N)/(2\lambda)$ and conjectured the matching
limit. Here we prove the matching lower bound, and therefore
\[
  \lim_{N\to\infty}\frac{\EE[\omega_{N,\lambda}]}{\log N}=\frac{1}{2\lambda}
  \qquad\text{for every fixed }\lambda>0 .
\]
\end{abstract}
\maketitle
\blfootnote{\emph{2020 Mathematics Subject Classification.} Primary 60K35, 82C24; Secondary 30C35, 60G42.
\emph{Key words and phrases.} Hastings--Levitov aggregation, Laplacian growth, tree-completion time,
coverage process, second moment method, Paley--Zygmund inequality, conformal slit map.}

\section{Introduction}
\label{sec:intro}

Diffusion-limited aggregation (DLA) \cite{WS} is a random growth model in which
particles arrive sequentially and attach to the existing cluster according to harmonic measure.
Hastings and Levitov \cite{HL} introduced an off-lattice analogue by representing planar growth
through iterated compositions of conformal maps, with each map encoding the addition of one
particle. The resulting Hastings--Levitov family is a tractable conformal model of Laplacian
growth that retains the screening and branching behavior characteristic of DLA.

Competition between branches is particularly transparent in cylindrical geometry. In the
cylindrical Hastings--Levitov$(0)$ process introduced by Procaccia and Zhuchenko \cite{PZ}, the
aggregate forms a forest of trees rooted on the base circle. Although new trees may continue to
emerge and finite trees may receive particles for a long time, there is almost surely a unique
tree that grows forever \cite{CPZ,NT}. This observation leads to a basic stabilization problem:
determine the \emph{one-arm domination time}, defined as the first time after which only the unique
infinite tree receives particles. The time at which competing arms cease to grow is a natural
quantitative measure of long-time stabilization in random growth models and was the principal
focus of \cite{CPZ}.

That study naturally suggests a complementary question about the creation, rather than the
subsequent growth, of competing trees: when does the base circle stop producing new trees? The
corresponding \emph{tree-completion time}, denoted by $\omega_{N,\lambda}$, is the last time at which
a new tree is born. In \cite{CPZ}, the sharp upper bound
\[
  \EE[\omega_{N,\lambda}]\le (1+\varepsilon)\frac{\log N}{2\lambda}
\]
was proved for every fixed $\lambda>0$, every $\varepsilon>0$, and all sufficiently large $N$, and the matching asymptotic
formula was conjectured. The present paper establishes the missing lower bound. Consequently,
\[
  \EE[\omega_{N,\lambda}]\sim\frac{\log N}{2\lambda},
  \qquad N\to\infty,
\]
which determines the expected tree-completion time to first order and resolves the conjecture
from \cite{CPZ}.

\subsection{The model}

We study the cylindrical Hastings--Levitov aggregation process with parameter $0$, introduced
by Procaccia and Zhuchenko \cite{PZ}, using the formulation of Chen, Procaccia and Zong
\cite{CPZ}. Fix an integer cylinder width $N\ge1$ and a particle size $\lambda>0$. We represent the
base circle and the cylinder by
\[
  \T_N=\R/(2\pi N\Z)=[0,2\pi N),\qquad
  \T^N=\{z\in\C:\Imag z>0\}/(z\sim z+2\pi N).
\]
We also write $\mathbb D_0=\{z\in\C:\abs z>1\}$ for the complement of the unit disk and
$\mathbb H=\{z\in\C:\Imag z>0\}$ for the upper half-plane. The attachment of a single particle of
size $\lambda$ at the base point $0$ is encoded by the \emph{cylindrical slit map centered at $0$},
which maps $\T^N$ conformally onto $\T^N\setminus[\,0,i\lambda\,]$. It is defined by
\[
  S^{N,\lambda}=f_N^{-1}\circ g^{-1}\circ\phi_\delta\circ g\circ f_N,
\]
where the constituent maps are
\[
  f_N(z)=e^{-iz/N},\qquad g(z)=i\frac{z-1}{z+1},\qquad \phi_a(z)=\sqrt{z^2(1-a^2)-a^2}\quad(a>0).
\]
The parameter $\delta=\delta(N,\lambda)\in(0,1)$ is determined by the normalization
$S^{N,\lambda}(0)=i\lambda$ and satisfies
\begin{equation}\label{eq:delta}
  \delta=1-\frac{2}{e^{\lambda/N}+1}=\frac{\lambda}{2N}+O\!\left(\frac{\lambda}{N}\right)^{3}
  \qquad(N\to\infty).
\end{equation}
The slit map centered at a general point $x\in\T_N$ is obtained from $S^{N,\lambda}$ by
translation,
\begin{equation}\label{eq:slitx}
  S^{N,\lambda}_x(z)=\operatorname{Re}(z)+S^{N,\lambda}(z-x)-\bigl(\operatorname{Re}(z-x)\bmod 2\pi N\bigr),
  \qquad z\in\T^N,
\end{equation}
so that $S^{N,\lambda}_0=S^{N,\lambda}$ and $S^{N,\lambda}_x$ attaches a slit of length $\lambda$
above $x$.
In the $\alpha=0$ model, the attachment angles are independent and uniform; neither
harmonic-measure feedback nor derivative-based normalization is used to rescale the particle size
(see \Cref{fig:chl}).

\begin{definition}[The process $\mathrm{CHL}_N$]\label{def:chl}
Let $P$ be a Poisson point process of intensity $1$ on $\R_+\times\T_N$. For every $t>0$,
the set $P\cap([0,t]\times\T_N)$ is almost surely finite; enumerate its points as
$(t_1,x_1),\dots,(t_n,x_n)$, where $0<t_1<\cdots<t_n\le t$. The \emph{cylindrical
Hastings--Levitov process with parameter $0$}, denoted by $\mathrm{CHL}_N$, is represented by the
aggregate map
\[
  \mathcal C^{N,\lambda}_t=S^{N,\lambda}_{x_1}\circ S^{N,\lambda}_{x_2}\circ\cdots\circ S^{N,\lambda}_{x_n}.
\]
Since $\abs{\T_N}=2\pi N$, particles arrive on the cylinder at total rate $2\pi N$.
\end{definition}

\begin{definition}[Trees]\label{def:tree}
At time $t$, the \emph{trees} of $\mathrm{CHL}_N$ are the connected components of
\[
  \mathcal C^{N,\lambda}_t(\T_N\times\{0\})\setminus(\T_N\times\{0\}).
\]
We write $\mathcal N_t$ for their number. The small-particle analysis of Norris and Turner
\cite{NT} implies that almost surely exactly one tree grows forever; we call it the infinite arm.
\end{definition}

\begin{figure}[t]
  \centering
  \includegraphics[width=0.55\linewidth]{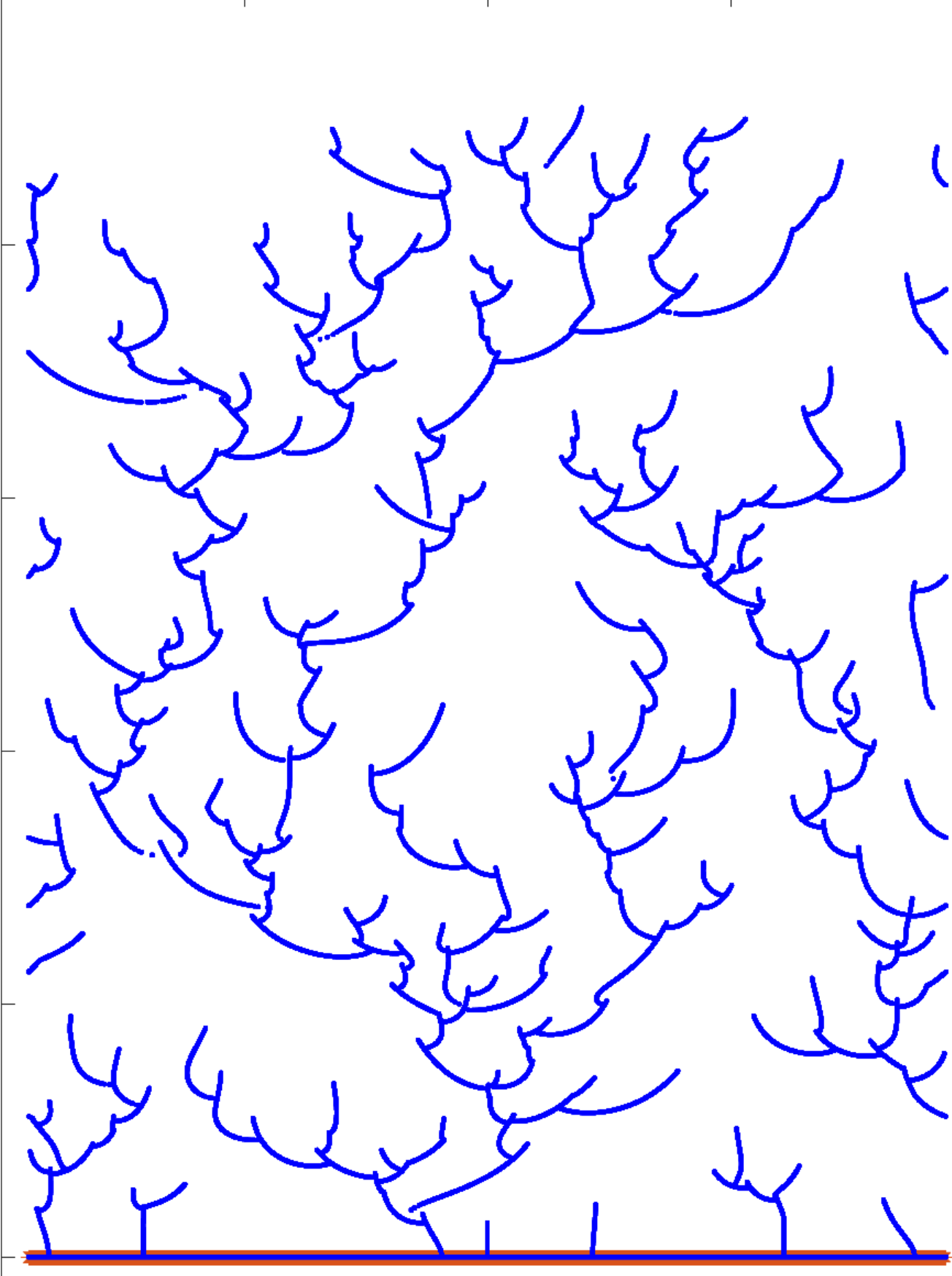}
  \caption{A simulation of $\mathrm{CHL}_N$ with the cylinder unrolled into a strip. The forest
  is rooted on the base circle $\T_N$, shown in red. Reproduced from Chen, Procaccia and Zong
  \cite[Figure~2 (left)]{CPZ}.}
  \label{fig:chl}
\end{figure}

We now formalize the tree-completion time introduced above.

\begin{definition}[Tree-completion time]\label{def:omega}
The \emph{tree-completion time} $\omega_{N,\lambda}$ is the last time at which a new tree is
born on the base circle $\T_N$. Equivalently, after time $\omega_{N,\lambda}$, every subsequently
attached particle joins a tree that already exists.
\end{definition}

\subsection{Main result}

\begin{theorem}[Sharp asymptotics]\label{thm:main}
For every fixed $\lambda>0$,
\begin{equation}\label{eq:limit}
  \lim_{N\to\infty}\frac{\EE[\omega_{N,\lambda}]}{\log N}=\frac{1}{2\lambda}.
\end{equation}
\end{theorem}

The upper bound in \eqref{eq:limit} was established in \cite[Thm.~1.5]{CPZ}: for every fixed
$\lambda>0$ and every $\varepsilon>0$,
\begin{equation}\label{eq:ub}
  \EE[\omega_{N,\lambda}]\le(1+\varepsilon)\frac{\log N}{2\lambda}
  \qquad\text{for all sufficiently large }N.
\end{equation}
The matching limit was conjectured in \cite[Rmk.~1.6]{CPZ}.  The main
contribution of this paper is the matching lower bound
\begin{equation}\label{eq:lb}
  \liminf_{N\to\infty}\frac{\EE[\omega_{N,\lambda}]}{\log N}\ge\frac{1}{2\lambda}.
\end{equation}
Combining \eqref{eq:lb} with \eqref{eq:ub} proves \Cref{thm:main} and resolves the conjecture.
The proof is given in \Cref{sec:reduction,sec:kernel,sec:proof}.

\subsection{Method}

Our proof begins with the reduction developed in \cite{CPZ}. This reduction couples the
growing forest to a Markov chain of \emph{marked configurations} and expresses
$\omega_{N,\lambda}$ in terms of a discrete \emph{coverage process}. We recall the formal
construction in \Cref{sec:reduction}; for the present discussion, it suffices to consider its
zero-colored marginal. On the rescaled circle $\T_1=\R/(2\pi\Z)$, let
$\Mz_k(0)\subseteq\T_1$ be the \emph{zero-colored set}, namely the union of base arcs that have not
yet generated a tree after the first $k$ particles have been placed, and set
\begin{equation}\label{eq:Zdef}
  Z_k:=\abs{\Mz_k(0)},\qquad Z_0=2\pi .
\end{equation}
A new tree is born at step $k+1$ exactly when the uniform position $x_{k+1}$ lies in
$\Mz_k(0)$. The recursion in \cite[Section~5]{CPZ} yields the exact first moment
\begin{equation}\label{eq:firstmoment}
  \EE Z_k=2\pi q^{k},\qquad q=1-\frac{a_\delta}{\pi},\qquad a_\delta=2\arctan\frac{\delta}{\sqrt{1-\delta^2}},
\end{equation}
By \eqref{eq:delta}, $a_\delta\sim\lambda/N$ and $1-q\sim\lambda/(\pi N)$. Hence the expected
zero-colored length reaches the scale $1/N$ when the number of particles is of order
$(\pi/\lambda)N\log N$. Converting particle number to physical time using the arrival rate
$2\pi N$ yields the scale $(\log N)/(2\lambda)$, and Markov's inequality gives the upper bound
\eqref{eq:ub}.

The first moment alone cannot yield a matching \emph{lower} bound: it shows that tree births
are unlikely after the critical scale, but it does not show that they persist up to that scale. To
prove \eqref{eq:lb}, we must show that new trees continue to appear until
$(1-o(1))(\pi/\lambda)N\log N$ particles have been added. Equivalently, we need a lower bound on the
last success time of the nonstationary coverage process. We obtain it by a second-moment argument.
Fix $\varepsilon\in(0,1)$, set
$m_N=\lfloor(1-\varepsilon)(\pi/\lambda)N\log N\rfloor$, and define
\begin{equation}\label{eq:YN}
  Y_N=\sum_{k\ge m_N}B_k,\qquad B_k=\one_{\{x_{k+1}\in\Mz_k(0)\}}.
\end{equation}
Thus $Y_N$ counts the tree births after the subcritical cutoff $m_N$. We prove that
$\EE Y_N\to\infty$ and $\EE Y_N^2\le(1+o(1))(\EE Y_N)^2$. The Paley--Zygmund inequality then gives
$\PP(Y_N>0)\to1$, and the resulting late birth yields the required lower bound after particle
number is converted to physical time.

The main technical input is a bound on the second moment of the zero-colored length. For a
deterministic set $A\subseteq\T_1$, let $\Sinv_x$ denote the zero-colored update associated with a
particle landing at $x$. The length after one step is
\[
  F_A(x)=\bigl|\Sinv_x(A)\bigr|=\int_A D_\delta(\theta-x)\,d\theta,
\]
where $D_\delta$, the inverse-slit derivative from \cite[\S2.1]{CPZ}, is given by
\[
D_{\delta}(u)
=
\sqrt{1-\delta^{2}}\,
\frac{\left|\tan\left(\frac{u}{2}\right)\right|}
{\sqrt{\tan^{2}\left(\frac{u}{2}\right)+\delta^{2}}},
\qquad
u\in\T_{1}\setminus\{0\}.
\]
Consequently, the one-step second-moment kernel is the quadratic functional
\[
  H(A)=\frac1{2\pi}\int_{\T_1}\!\Bigl(\int_A D_\delta(\theta-x)\,d\theta\Bigr)^2 dx
  =\EE\bigl[Z_{k+1}^2\mid\mathcal F_k\bigr]\big|_{A=\Mz_k(0)} .
\]
Because $F_A=D_\delta*\one_A$ is a circular convolution, Parseval's identity gives
\[
  H(A)=4\pi^2\sum_{n\in\Z}\widehat D_\delta(n)^2\,\abs{\widehat{\one_A}(n)}^2;
\]
see \Cref{lem:kernel}. The zero Fourier mode $\widehat D_\delta(0)=q$ yields the first-moment
term $q^2\abs{A}^2$. A deterministic contraction of the form $H(A)\le q^2\abs{A}^2$ would be
tempting, but it fails for every nontrivial $A$ (\Cref{lem:excess}). Instead, we prove the sharp
\emph{excess} estimate
\begin{equation}\label{eq:excessintro}
  0\le H(A)-q^2\abs{A}^2\le\frac{2a_\delta^2}{\pi}\abs{A}.
\end{equation}
The excess is precisely the contribution of the nonzero Fourier modes:
\[
  H(A)-q^2\abs{A}^2
  =4\pi^2\sum_{n\ne0}\widehat D_\delta(n)^2\,\abs{\widehat{\one_A}(n)}^2
  =\Var_x(F_A).
\]
The mass identity $\int_{\T_1}(1-D_\delta)=2a_\delta$ implies that this quantity is
$O(a_\delta^2)\abs{A}$; see \Cref{lem:excess}. Since $a_\delta^2=O(N^{-2})$, the excess is
negligible relative to $q^2\abs{A}^2$ at the lower-bound scale. Iterating
\eqref{eq:excessintro} yields the summed $L^2$ estimate
$\sum_{k\ge m_N}\EE Z_k^2\le(1+o(1))\sum_{k\ge m_N}(\EE Z_k)^2$, which closes the second-moment
argument. From the martingale perspective of \Cref{rem:mart}, the same mechanism shows that
$M_k=Z_k/q^k$ is a nonnegative martingale whose renormalized predictable quadratic variation has
expected size $o(1)$ up to the critical time.

\subsection{Organization and conventions}

In \Cref{sec:reduction}, we recall the coverage process and the marked-configuration
formalism from \cite{CPZ}, compute the first two moments of the late-birth count $Y_N$, and reduce
\eqref{eq:lb} to a summed $L^2$ estimate for $Z_k$. In \Cref{sec:kernel}, we derive the convolution
formula for the one-step kernel $H(A)$, prove the excess bound \eqref{eq:excessintro}, and obtain the
required summed $L^2$ estimate. Finally, \Cref{sec:proof} completes the Paley--Zygmund argument and
proves \Cref{thm:main}.

Throughout the paper, $\lambda>0$ is fixed and $N\to\infty$. Unless stated otherwise,
implicit constants and the notation $o(\cdot)$ and $O(\cdot)$ may depend on $\lambda$ (and on
$\varepsilon$ when it appears). We write $\F_k=\sigma(x_1,\dots,x_k)$ for the filtration generated
by the particle positions, $\abs{\cdot}$ for Lebesgue length on $\T_1$, and $\one_E$ for the
indicator of an event $E$. We use throughout the notation
\begin{equation}\label{eq:notation}
  a_\delta=2\arctan\frac{\delta}{\sqrt{1-\delta^2}},\qquad q=1-\frac{a_\delta}{\pi}\in(0,1),
\end{equation}
Then \eqref{eq:delta} gives
\begin{equation}\label{eq:1mq}
  a_\delta=\pi(1-q),\qquad 1-q=\frac{\lambda}{\pi N}+O_\lambda(N^{-3}).
\end{equation}

\section{The coverage process and the second-moment scheme}
\label{sec:reduction}

\subsection{The zero-colored set}

We recall the reduction of tree completion to a coverage process, following
\cite[\S2.1, \S3, \S5]{CPZ}. The coupling runs through the marked-configuration formalism of
\cite[\S3]{CPZ}, which we now state.

\begin{definition}[Marked configuration]\label{def:mark}
A \emph{marked configuration} $\mathtt M$ on $\T_1$ is an ordered tuple of closed arcs together with a
coloring,
\[
  \mathtt M=\bigl[(I_1,\dots,I_k),(a_1,\dots,a_k)\bigr],
\]
the colors $a_1,\dots,a_k$ being distinct nonzero integers; we write $\mathtt M(a_j)=I_j$ and say that
$I_j$ is marked by color $a_j$. The complementary set
$I_0(\mathtt M)=\overline{\T_1\setminus\bigcup_{j=1}^{k}I_j}$ is a finite union of disjoint closed
arcs, each assigned color $0$; we write $\mathtt M(0)=I_0(\mathtt M)$ and call it the
\emph{zero-colored set} of $\mathtt M$, and set $\mathtt n(\mathtt M)=k$ for its number of nonzero
colors.
\end{definition}

\begin{definition}[Marked configuration sequence]\label{def:markseq}
The \emph{marked configuration sequence} $\{\mathtt M_k\}_{k\ge0}$ generated by driving positions
$x_1,x_2,\dots\in\T_1$ is given by $\mathtt M_0=\emptyset$ (a single arc of color $0$) and
$\mathtt M_{k+1}=S^{\mathtt{inv}}_{x_{k+1}}(\mathtt M_k)$, where on
$\mathtt M_k=[(I_1,\dots,I_k),(1,\dots,k)]$
\[
  S^{\mathtt{inv}}_{x}(\mathtt M_k)=
  \begin{cases}
    \bigl[(S^{\mathtt{inv}}_{x}I_1,\dots,S^{\mathtt{inv}}_{x}I_k),\,(1,\dots,k)\bigr],
      & x\in\bigcup_{j}I_j,\\[4pt]
    \bigl[(S^{\mathtt{inv}}_{x}I_1,\dots,S^{\mathtt{inv}}_{x}I_k,\,S^{\mathtt{inv}}_{x}\{x\}),\,(1,\dots,k+1)\bigr],
      & x\notin\bigcup_{j}I_j,
  \end{cases}
\]
and $S^{\mathtt{inv}}_x$ is the inverse slit map of \cite[\S2.1]{CPZ} (its derivative $D_\delta$ is
recorded in \eqref{eq:Ddef} below). By \cite[\S3]{CPZ} this sequence couples to $\mathrm{CHL}_N$ so
that each nonzero color tracks the inverse harmonic-measure interval of one tree; hence
$\mathtt n(\mathtt M_k)$ is the number of trees after $k$ particles, and a new color --- a new tree
--- appears at step $k+1$ exactly when $x_{k+1}$ falls in $\mathtt M_k(0)$.
\end{definition}

The coverage process is the zero-colored marginal $\Mz_k(0)$. After $k$ particles have
been placed, the zero-colored set $\Mz_k(0)\subseteq\T_1$ --- the union of arcs of the (rescaled) base
circle from which no tree has yet grown --- has length $Z_k=\abs{\Mz_k(0)}$ as in \eqref{eq:Zdef}. A
new tree is born at step $k+1$
precisely when the uniformly placed position $x_{k+1}\in\T_1$ falls in $\Mz_k(0)$, so with $B_k$ as
in \eqref{eq:YN},
\begin{equation}\label{eq:birthprob}
  \PP(B_k=1\mid\F_k)=\frac{Z_k}{2\pi}.
\end{equation}
When a particle lands at $x$, the zero-colored set is updated by the broadened inverse slit map
$\Sinv_x$: every zero arc $J$ with $x\notin J$ is pulled back to $S^{\mathtt{inv}}_x(J)$, while the arc
$J=[a,b]$ with $x\in J$ is split into the two side arcs
\[
  \Sinv_x(J)=\bigl[S^{\mathtt{inv}}_x(a),\,x-a_\delta\bigr]\sqcup\bigl[x+a_\delta,\,S^{\mathtt{inv}}_x(b)\bigr],
\]
the newly born interval of length $2a_\delta$ around $x$ being deleted (it becomes a new nonzero
color); summing over the zero arcs gives $\Mz_{k+1}(0)=\Sinv_{x_{k+1}}(\Mz_k(0))$. The Section~5
recursion of \cite{CPZ} states that
\begin{equation}\label{eq:recursion}
  \EE\bigl[Z_{k+1}\mid\F_k\bigr]=q\,Z_k,\qquad\text{hence}\qquad \EE Z_k=2\pi q^{k},
\end{equation}
which is \eqref{eq:firstmoment}. Finally, writing $T_{\mathrm{tree}}$ for the index of the last
particle that creates a new tree, one has $\omega_{N,\lambda}=t_{T_{\mathrm{tree}}}$, and since the
arrival times are an independent rate-$2\pi N$ Poisson process,
\begin{equation}\label{eq:count}
  \EE[\omega_{N,\lambda}]=\frac{\EE[T_{\mathrm{tree}}]}{2\pi N}.
\end{equation}
Formulae \eqref{eq:birthprob}--\eqref{eq:count} are exactly \cite[\S5]{CPZ} and are used as a black
box. We also record the pointwise contraction of the zero set, an immediate consequence of the
inverse-slit interval estimates of \cite[\S2.1]{CPZ}.

\begin{lemma}[Deterministic monotonicity]\label{lem:mono}
Almost surely, $Z_{k+1}\le Z_k$ for every $k\ge0$.
\end{lemma}

\begin{proof}
It suffices to bound the update of a single zero arc $J$ and sum. If the particle position $x\notin
J$, the inverse-slit estimate of \cite[\S2.1]{CPZ} gives $\abs{S_x^{\mathtt{inv}}(J)}\le
\sqrt{1-\delta^2}\,\abs{J}\le\abs{J}$. If $x\in J$, the same section gives
$\abs{S_x^{\mathtt{inv}}(J)}\le 2a_\delta+(1-\tfrac12\delta^2)\abs{J}$, and the zero-colored update
removes the newly born interval $S_x^{\mathtt{inv}}(\{x\})$ of length $2a_\delta$, so
$\bigl|\Sinv_x(J)\bigr|\le(1-\tfrac12\delta^2)\abs{J}\le\abs{J}$. Summing over the finitely many
disjoint zero arcs comprising $\Mz_k(0)$ gives $Z_{k+1}\le Z_k$.
\end{proof}

\subsection{First and second moments of the late-birth count}

We fix $\varepsilon\in(0,1)$ and the subcritical cutoff
\begin{equation}\label{eq:mN}
  m_N=\left\lfloor(1-\varepsilon)\frac{\pi}{\lambda}N\log N\right\rfloor,
\end{equation}
and study $Y_N=\sum_{k\ge m_N}B_k$ from \eqref{eq:YN}.

\begin{proposition}[First moment]\label{prop:first}
With $m_N$ as in \eqref{eq:mN},
\[
  \EE Y_N=\frac{q^{m_N}}{1-q}=\frac{\pi}{\lambda}N^{\varepsilon}(1+o(1))\xrightarrow[N\to\infty]{}\infty .
\]
\end{proposition}

\begin{proof}
By \eqref{eq:birthprob} and \eqref{eq:recursion}, $\EE B_k=\EE Z_k/2\pi=q^{k}$, hence
$\EE Y_N=\sum_{k\ge m_N}q^{k}=q^{m_N}/(1-q)$. From \eqref{eq:1mq}, $m_N(1-q)=(1-\varepsilon)\log
N+o(1)$, so
\[
  q^{m_N}=\exp\{m_N\log(1-(1-q))\}=\exp\{-(1-\varepsilon)\log N+o(1)\}=N^{-(1-\varepsilon)}(1+o(1)).
\]
Since $1/(1-q)=(\pi N/\lambda)(1+o(1))$, we obtain
$\EE Y_N=N^{-(1-\varepsilon)}(1+o(1))\cdot(\pi N/\lambda)(1+o(1))=(\pi/\lambda)N^{\varepsilon}(1+o(1))$.
\end{proof}

\begin{proposition}[Second-moment identity]\label{prop:second}
For every deterministic cutoff $m$, with $Y_m=\sum_{k\ge m}B_k$,
\begin{equation}\label{eq:secondid}
  \EE Y_m^2=\EE Y_m+\frac{1}{\pi(1-q)}\sum_{k\ge m}\EE\bigl[B_kZ_{k+1}\bigr].
\end{equation}
\end{proposition}

\begin{proof}
We work with the truncations $Y_{m,L}=\sum_{k=m}^{L}B_k$ and let $L\to\infty$ at the end; since
$0\le Y_{m,L}\uparrow Y_m$, monotone convergence gives $\EE Y_{m,L}\to\EE Y_m$ and $\EE
Y_{m,L}^2\to\EE Y_m^2$, so it suffices to prove \eqref{eq:secondid} for each finite $L$ and pass to
the limit (all sums below then being finite).

Let $m\le k<l\le L$. Conditioning on $\F_{k+1}$ and applying the first-moment recursion
\eqref{eq:recursion} to the evolution after step $k+1$ gives $\EE[Z_l\mid\F_{k+1}]=q^{\,l-k-1}Z_{k+1}$
--- this is licensed by the configuration-independence of the rate $q$ in \eqref{eq:recursion}, the
future driving positions being independent uniform samples --- whence by \eqref{eq:birthprob},
\[
  \EE[B_l\mid\F_{k+1}]=\frac{\EE[Z_l\mid\F_{k+1}]}{2\pi}=\frac{q^{\,l-k-1}Z_{k+1}}{2\pi}.
\]
Because $B_k=\one_{\{x_{k+1}\in\Mz_k(0)\}}$ is $\F_{k+1}$-measurable (it depends on $x_1,\dots,x_{k+1}$),
the tower property yields
\[
  \EE[B_kB_l]=\EE\bigl[B_k\,\EE(B_l\mid\F_{k+1})\bigr]=\frac{q^{\,l-k-1}}{2\pi}\,\EE[B_kZ_{k+1}].
\]
Therefore, summing the off-diagonal terms and using
$\sum_{r\ge1}q^{\,r-1}=1/(1-q)$,
\[
  \EE Y_{m,L}^2=\sum_{k=m}^{L}\EE B_k+2\!\!\sum_{m\le k<l\le L}\!\!\EE[B_kB_l]
  =\EE Y_{m,L}+\frac1{\pi}\sum_{k=m}^{L}\EE[B_kZ_{k+1}]\sum_{r=1}^{L-k}q^{\,r-1}.
\]
Letting $L\to\infty$ (monotone convergence, and $\sum_{r=1}^{L-k}q^{r-1}\uparrow1/(1-q)$) gives
\eqref{eq:secondid}.
\end{proof}

The identity \eqref{eq:secondid} exposes the only nontrivial quantity, the birth--zero correlation
$\sum_k\EE[B_kZ_{k+1}]$. By \Cref{lem:mono} it is controlled by the plain second moment of $Z_k$:
since $Z_{k+1}\le Z_k$ and $\EE[B_k\mid\F_k]=Z_k/2\pi$,
\begin{equation}\label{eq:corrbound}
  \EE[B_kZ_{k+1}]\le\EE[B_kZ_k]=\EE\bigl[Z_k\,\EE(B_k\mid\F_k)\bigr]=\frac1{2\pi}\EE[Z_k^2].
\end{equation}
Thus the second moment of $Y_N$ is governed by the summed $L^2$-norm $\sum_{k\ge m_N}\EE Z_k^2$,
which we estimate in the next section.

\section{The one-step second-moment kernel and the summed \texorpdfstring{$L^2$}{L2} estimate}
\label{sec:kernel}

Throughout this section $A=\bigsqcup_{i=1}^{\ell}J_i\subset\T_1$ is a deterministic finite union of
arcs. We first record the exact one-step second-moment kernel, then prove that its deterministic
excess over the first-moment rate is small, and finally iterate to a summed $L^2$ bound.

\subsection{The convolution formula}

Following \cite[\S2.1]{CPZ}, the inverse slit map has derivative $D_\delta$, where
\begin{equation}\label{eq:Ddef}
  D_\delta(u)=\sqrt{1-\delta^2}\,\frac{\abs{\tan(u/2)}}{\sqrt{\tan^2(u/2)+\delta^2}}
  \quad(u\in\T_1\setminus\{0\}),\qquad D_\delta(0)=0,
\end{equation}
so that $\tfrac{d}{d\theta}S_0^{\mathtt{inv}}(\theta)=D_\delta(\theta)$ and, by rotation
invariance, $\tfrac{d}{d\theta}S_x^{\mathtt{inv}}(\theta)=D_\delta(\theta-x)$. Note $0\le
D_\delta\le\sqrt{1-\delta^2}\le1$, $D_\delta$ is even about $0$ (i.e.\ $D_\delta(2\pi-u)=D_\delta(u)$),
and
\begin{equation}\label{eq:Dmass}
  \int_0^{2\pi}D_\delta(u)\,du=2\pi-2a_\delta=2\pi q .
\end{equation}
For a particle at $x$, write $F_A(x)=\bigl|\Sinv_x(A)\bigr|$ for the length of the zero set after one
step, and define the one-step $L^2$ kernel
\begin{equation}\label{eq:Hdef}
  H(A)=\frac1{2\pi}\int_{\T_1}F_A(x)^2\,dx .
\end{equation}
Since $x_{k+1}$ is uniform on $\T_1$ and $\Mz_{k+1}(0)=\Sinv_{x_{k+1}}(\Mz_k(0))$, taking
$A=\Mz_k(0)$ gives the probabilistic meaning
\begin{equation}\label{eq:Hprob}
  \EE\bigl[Z_{k+1}^2\mid\F_k\bigr]=H\bigl(\Mz_k(0)\bigr).
\end{equation}

\begin{lemma}[Kernel formula]\label{lem:kernel}
For every finite union $A\subset\T_1$, and up to endpoint null sets,
\begin{equation}\label{eq:FA}
  F_A(x)=\int_A D_\delta(\theta-x)\,d\theta,\qquad x\in\T_1 .
\end{equation}
Consequently, with the Fourier normalization
$\widehat f(n)=\tfrac1{2\pi}\int_{\T_1}f(x)e^{-inx}\,dx$,
\begin{equation}\label{eq:Hfourier}
  H(A)=4\pi^2\sum_{n\in\Z}\widehat D_\delta(n)^2\,\bigl|\widehat{\one_A}(n)\bigr|^2
  =q^2\abs{A}^2+4\pi^2\sum_{n\ne0}\widehat D_\delta(n)^2\,\bigl|\widehat{\one_A}(n)\bigr|^2 .
\end{equation}
\end{lemma}

\begin{proof}
Let \(A=\bigsqcup_{i=1}^{m}J_i\), where  \(J_i\) is an arc on $\mathbb T_1$.  For \(x\notin A\), the derivative formula gives
\[
\bigl|S_x^{\mathtt{inv}}(J_i)\bigr|
=
\int_{J_i}D_\delta(\theta-x)\,d\theta.
\]
If \(x\in J_\ell=[\alpha_\ell,\beta_\ell]\), choose
\(\alpha_\ell<x<\beta_\ell\).  The zero-colored update splits \(J_\ell\) into the two
side intervals, and hence
\[
\begin{aligned}
\bigl|\widehat S_x^{\mathtt{inv}}(J_\ell)\bigr|
&=
\bigl|[S_x^{\mathtt{inv}}(\alpha_\ell),x-a_\delta]\bigr|
+\bigl|[x+a_\delta,S_x^{\mathtt{inv}}(\beta_\ell)]\bigr|\\
&=
\int_{\alpha_\ell}^{x}D_\delta(\theta-x)\,d\theta
+\int_x^{\beta_\ell}D_\delta(\theta-x)\,d\theta
=
\int_{J_\ell}D_\delta(\theta-x)\,d\theta.
\end{aligned}
\]
Since the inverse boundary map preserves cyclic order on the cut circle, the images of
distinct arcs are disjoint up to endpoints.  Summing their lengths therefore yields
\eqref{eq:FA}.

For \eqref{eq:Hfourier}, \eqref{eq:FA} states that $F_A=D_\delta*\one_A$ is the circular convolution
$(D_\delta*\one_A)(x)=\int_{\T_1}D_\delta(x-y)\one_A(y)\,dy$ (using that $D_\delta$ is even), whence
$\widehat{F_A}(n)=2\pi\,\widehat D_\delta(n)\,\widehat{\one_A}(n)$. Parseval's identity
$\tfrac1{2\pi}\int_{\T_1}\abs{F_A}^2=\sum_n\bigl|\widehat{F_A}(n)\bigr|^2$ then gives the first
equality in \eqref{eq:Hfourier}. By \eqref{eq:Dmass}, $\widehat D_\delta(0)=q$, and
$\widehat{\one_A}(0)=\abs{A}/2\pi$, so the $n=0$ term equals
$4\pi^2q^2(\abs{A}/2\pi)^2=q^2\abs{A}^2$; this is the second equality.
\end{proof}

\subsection{The deterministic excess bound}

\begin{lemma}[Excess bound]\label{lem:excess}
For every finite union $A\subset\T_1$,
\begin{equation}\label{eq:excess}
  0\le H(A)-q^2\abs{A}^2\le\frac{2a_\delta^2}{\pi}\abs{A}.
\end{equation}
The lower bound is an equality only when $D_\delta*\one_A$ is constant; in particular the
deterministic contraction $H(A)\le q^2\abs{A}^2$ \emph{fails} in general --- for instance for every
sufficiently short arc.
\end{lemma}

\begin{proof}
By \eqref{eq:Hfourier}, $H(A)-q^2\abs{A}^2=4\pi^2\sum_{n\ne0}\widehat D_\delta(n)^2
\bigl|\widehat{\one_A}(n)\bigr|^2\ge0$, with equality iff every nonzero Fourier coefficient of
$F_A=D_\delta*\one_A$ vanishes, i.e.\ iff $F_A$ is constant. Writing $z=\abs{A}$, this excess is the
variance of $F_A$ under a uniform sample $x$:
\[
  H(A)-q^2z^2=\frac1{2\pi}\int_{\T_1}F_A^2-\Bigl(\frac1{2\pi}\int_{\T_1}F_A\Bigr)^2=\Var_{x\sim\Unif(\T_1)}\!\bigl(F_A(x)\bigr),
\]
since $\tfrac1{2\pi}\int F_A=q z$ by \eqref{eq:Dmass}.

To bound the variance, set $E_\delta(u)=1-D_\delta(u)$, so $0\le E_\delta\le1$ and, by
\eqref{eq:Dmass}, $\int_{\T_1}E_\delta=2\pi-2\pi q=2a_\delta$. With
$G_A(x)=\int_A E_\delta(\theta-x)\,d\theta$ we have $F_A(x)=z-G_A(x)$, hence
$\Var(F_A)=\Var(G_A)\le\EE[G_A^2]$. Now $0\le G_A(x)\le\int_{\T_1}E_\delta=2a_\delta$ pointwise, and
\[
  \EE[G_A]=\frac1{2\pi}\int_{\T_1}G_A(x)\,dx=\frac{z}{2\pi}\int_{\T_1}E_\delta=\frac{a_\delta}{\pi}z,
\]
so, using $G_A^2\le 2a_\delta\,G_A$,
\[
  \EE[G_A^2]\le 2a_\delta\,\EE[G_A]=2a_\delta\cdot\frac{a_\delta}{\pi}z=\frac{2a_\delta^2}{\pi}\abs{A},
\]
which is \eqref{eq:excess}.

Finally, for a short arc $A_L=[0,L]$, \eqref{eq:FA} gives
\[
  \frac{H(A_L)-q^2L^2}{L^2}\xrightarrow[L\downarrow0]{}\frac1{2\pi}\int_{\T_1}D_\delta(u)^2\,du-q^2>0,
\]
the limit being positive because $D_\delta$ is nonconstant (so $\tfrac1{2\pi}\int
D_\delta^2>(\tfrac1{2\pi}\int D_\delta)^2=q^2$ by Cauchy--Schwarz). Hence $H(A_L)>q^2L^2$ for all
small $L$, and the deterministic contraction fails.
\end{proof}

\subsection{The summed \texorpdfstring{$L^2$}{L2} estimate}

\begin{proposition}[Summed $L^2$ bound]\label{prop:L2}
With $m_N$ as in \eqref{eq:mN},
\begin{equation}\label{eq:L2}
  \sum_{k\ge m_N}\EE Z_k^2\le(1+o(1))\,4\pi^2\sum_{k\ge m_N}q^{2k}=(1+o(1))\sum_{k\ge m_N}(\EE Z_k)^2 .
\end{equation}
\end{proposition}

\begin{proof}
Combining \eqref{eq:Hprob} with \Cref{lem:excess} applied to $A=\Mz_k(0)$ (so $\abs{A}=Z_k$),
\[
  \EE\bigl[Z_{k+1}^2\mid\F_k\bigr]=H(\Mz_k(0))\le q^2Z_k^2+\frac{2a_\delta^2}{\pi}Z_k .
\]
Taking expectations and using $\EE Z_k=2\pi q^{k}$, the sequence $s_k:=\EE Z_k^2$ obeys
$s_{k+1}\le q^2 s_k+4a_\delta^2 q^{k}$. Since $s_0=Z_0^2=4\pi^2$, iteration gives
\[
  s_k\le 4\pi^2 q^{2k}+4a_\delta^2\sum_{j=0}^{k-1}q^{\,2(k-1-j)}q^{\,j}
  \le 4\pi^2 q^{2k}+\frac{4a_\delta^2}{1-q}\,q^{\,k-1},
\]
the last step bounding $\sum_{j=0}^{k-1}q^{\,2(k-1-j)+j}=q^{\,k-1}\sum_{j=0}^{k-1}q^{\,k-1-j}
\le q^{\,k-1}/(1-q)$. Summing over $k\ge m_N$,
\[
  \sum_{k\ge m_N}s_k\le 4\pi^2\sum_{k\ge m_N}q^{2k}+\frac{4a_\delta^2}{1-q}\sum_{k\ge m_N}q^{\,k-1}
  =4\pi^2\,\frac{q^{2m_N}}{1-q^2}+\frac{4a_\delta^2\,q^{\,m_N-1}}{(1-q)^2}.
\]
Using $a_\delta=\pi(1-q)$ from \eqref{eq:1mq}, the ratio of the error term to the main term is
\[
  \frac{4a_\delta^2 q^{\,m_N-1}/(1-q)^2}{4\pi^2 q^{2m_N}/(1-q^2)}=(1-q^2)\,q^{-m_N-1}.
\]
By \eqref{eq:1mq}, $1-q^2=(1-q)(1+q)=O(N^{-1})$, while $q^{-m_N}=N^{1-\varepsilon}(1+o(1))$ from the
proof of \Cref{prop:first}. Hence the ratio is $O(N^{-1})\,O(N^{1-\varepsilon})=O(N^{-\varepsilon})=o(1)$,
which proves \eqref{eq:L2}; the final equality uses $(\EE Z_k)^2=4\pi^2 q^{2k}$.
\end{proof}

\section{Proof of the main theorem}
\label{sec:proof}

\begin{proposition}[The lower bound]\label{prop:lb}
For every fixed $\lambda>0$,
\[
  \liminf_{N\to\infty}\frac{\EE[\omega_{N,\lambda}]}{\log N}\ge\frac{1}{2\lambda}.
\]
\end{proposition}

\begin{proof}
Fix $\varepsilon\in(0,1)$ and $m_N$ as in \eqref{eq:mN}. Starting from the second-moment identity
\eqref{eq:secondid} and the correlation bound \eqref{eq:corrbound},
\[
  \EE Y_N^2\le\EE Y_N+\frac{1}{2\pi^2(1-q)}\sum_{k\ge m_N}\EE Z_k^2 .
\]
By the summed $L^2$ estimate \eqref{eq:L2} and $\sum_{k\ge m_N}(\EE Z_k)^2=4\pi^2q^{2m_N}/(1-q^2)$,
\[
  \EE Y_N^2\le\EE Y_N+(1+o(1))\,\frac{2\,q^{2m_N}}{(1-q)(1-q^2)} .
\]
Since $1-q^2=(1-q)(1+q)\sim2(1-q)$, the last term is
$(1+o(1))\,q^{2m_N}/(1-q)^2=(1+o(1))(\EE Y_N)^2$ by \Cref{prop:first}. As $\EE Y_N\to\infty$, the
linear term $\EE Y_N$ is negligible against $(\EE Y_N)^2$, so
\[
  \EE Y_N^2\le(1+o(1))(\EE Y_N)^2 .
\]
The Paley--Zygmund inequality \cite{Durrett} (equivalently, Cauchy--Schwarz applied to
$Y_N=Y_N\one_{\{Y_N>0\}}$) then gives
\[
  \PP(Y_N>0)\ge\frac{(\EE Y_N)^2}{\EE Y_N^2}\xrightarrow[N\to\infty]{}1 .
\]
If $Y_N>0$, some $B_k=1$ with $k\ge m_N$, i.e.\ a new tree is born at particle index at least
$m_N+1$, so $T_{\mathrm{tree}}\ge m_N+1$. Hence, by \eqref{eq:count},
\[
  \EE[\omega_{N,\lambda}]=\frac{\EE[T_{\mathrm{tree}}]}{2\pi N}
  \ge\frac{m_N\,\PP(Y_N>0)}{2\pi N}=(1-o(1))(1-\varepsilon)\frac{\log N}{2\lambda},
\]
where we used $m_N=(1-\varepsilon)(\pi/\lambda)N\log N\,(1+o(1))$. Letting $\varepsilon\downarrow0$
gives the claim.
\end{proof}

\begin{proof}[Proof of \Cref{thm:main}]
The upper bound \eqref{eq:ub} is \cite[Thm.~1.5]{CPZ}, so
$\limsup_{N}\EE[\omega_{N,\lambda}]/\log N\le1/(2\lambda)$. \Cref{prop:lb} gives the matching
$\liminf\ge1/(2\lambda)$. Together they yield \eqref{eq:limit}.
\end{proof}

\begin{remark}[Martingale reformulation]\label{rem:mart}
The mechanism behind the lower bound admits a martingale reading. By \eqref{eq:recursion} the
process $M_k=Z_k/q^{k}$ is a nonnegative $\F_k$-martingale, with increments
$M_{k+1}-M_k=(Z_{k+1}-qZ_k)/q^{k+1}$; by \eqref{eq:Hprob} and \Cref{lem:excess},
\[
  \EE\bigl[(M_{k+1}-M_k)^2\mid\F_k\bigr]=q^{-2(k+1)}\bigl(H(\Mz_k(0))-q^2Z_k^2\bigr)
  \le q^{-2(k+1)}\frac{2a_\delta^2}{\pi}Z_k .
\]
Hence the predictable quadratic variation of $M$, summed up to the critical time, is $o(1)$ in
expectation:
\[
  \EE\sum_{k<m_N}\EE\bigl[(M_{k+1}-M_k)^2\mid\F_k\bigr]
  \le\frac{2a_\delta^2}{\pi}\sum_{k<m_N}q^{-2(k+1)}\,\EE Z_k
  =4\pi a_\delta\,q^{-m_N}(1+o(1))=O(N^{-\varepsilon}),
\]
using $\EE Z_k=2\pi q^{k}$, $a_\delta=\pi(1-q)$ and $q^{-m_N}=N^{1-\varepsilon}(1+o(1))$. With Doob's inequality,
$M_k$ is essentially frozen at its mean below the critical scale, which is the same
non-fragmentation statement that drives \Cref{prop:lb}.
\end{remark}

\section*{Acknowledgments}

\noindent The main results of this paper were obtained by Eureka and subsequently
verified by the authors. Eureka is a multi-agent system developed by JIUCHONG at the University of
Science and Technology of China for mathematical research through human--AI interaction. The
authors assume full responsibility for the content of this paper.

\end{document}